\documentclass[10pt,leqno]{amsart}
\usepackage{graphicx}
\usepackage{indentfirst,csquotes}

\usepackage{amssymb,amsthm,amsmath,mathtools}
\usepackage{xcolor,paralist,hyperref,fancyhdr,etoolbox}
\usepackage[english]{babel}
\newtheorem{theorem}{Theorem}[section]
\newtheorem{definition}[theorem]{Definition}
\newtheorem{example}[theorem]{Example}
\newtheorem{lemma}[theorem]{Lemma}

\newtheorem{corollary}[theorem]{Corollary}

\theoremstyle{remark}
\newtheorem{remark}[theorem]{Remark}
\numberwithin{equation}{section}

\hypersetup{ colorlinks=true, linkcolor=blue, filecolor=blue, urlcolor=blue }

\usepackage[doi,citestyle=ieee,backend=biber,sorting=nty,maxbibnames=10]{biblatex}
\bibliography{bibliographyFaber-Krahn.bib}
\AtEveryBibitem{
 \clearfield{date}
 \clearfield{eprint}
 \clearfield{isbn}
 \clearfield{issn}
 \clearlist{location}
 \clearfield{month}
 \clearfield{doi}
\clearfield{url}
\clearfield{zbmath}
\clearfield{zbl}
 \ifentrytype{book}{}{
  \clearlist{publisher}
  \clearname{editor}
  \clearfield{booktitle}
 }
}
\usepackage{mdframed}
\usepackage{comment}
\begin{document}
\title[Bergman spaces, convex functionals, concentration estimates and stability]{Stability of Wehrl-type Functionals and Concentration Estimates on Bergman Spaces of Log-Subharmonic Functions on the Unit Sphere} 
\author{Vladan Jaguzović}
\address{ Faculty of Natural Sciences and Mathematics,
  University of Banja Luka,
  Mladena Stojanovića 2,
  78000 Banja Luka,
  Republic of Srpska,
  Bosnia and Herzegovina
}
\email{vladan.jaguzovic@pmf.unibl.org}

\author{Petar Melentijević}
\address{Faculty of Mathematics,
  University of Belgrade,
  Studentski trg 16,
  11000 Beograd,
  Republic of Serbia}
\email{petar.melentijevic@matf.bg.ac.rs}
\begin{NoHyper} 
  \let\thefootnote\relax
\end{NoHyper}
\begin{abstract}
  In this paper, we consider weighted Bergman classes $\mathcal{B}_{\alpha,p}$ of real-analytic functions on $\mathbb{R}^{n}$, $n\geq 2$ whose moduli are log-subharmonic with respect to the spherical metric. Using the isoperimetric inequality for the spherical metric we prove a certain monotonicity property for super-level sets of $|f(x)|^p\mathcal{W}_n^{\alpha}(x),$ where $f\in \mathcal{B}_{\alpha,p}$ and $\mathcal{W}_n^{\alpha}(x)$ is the Bergman weight. As a consequence, we obtain sharp concentration estimates and solve maximization problems for convex Wehrl-type functionals. We also establish quantitative stability estimates, including a bound for the distance of \(|f|\) from spherical comparison profiles in terms of the concentration deficit.
\end{abstract} 
\maketitle \pagestyle{myheadings}
\section{Introduction}\label{sec:introduction}

Recently, there has been growing interest in different questions regarding  Wehrl-type entropy conjectures, concentration estimates and/or maximizing certain functionals on Bergman, Fock and polynomial spaces. In their prominent paper \cite{NicolaTilli2022}, Nicola and Tilli proved the Faber-Krahn inequality for the short-time Fourier transform via an ingenious method of obtaining a differential inequality for super-level sets of certain functions using the isoperimetric inequality. Very quickly, Kulikov (\cite{kulikov2022functionals-with-extrema}) successfully adapted this technique to the hyperbolic setting on the unit disk and solved the conjectures of Lieb-Solovej for $SU(1,1)$, Pavlovi\'{c} and Brevig-Ortega-Cerd\`a-Seip-Zhao (\cite{LiebSolovej}, \cite{Pavlovic}, \cite{BrevigOSSeipZhao}). The concentration estimate for the weighted Bergman spaces on the unit disk was proved by Ramos and Tilli (\cite{zbMATH07738115}). Frank (\cite{zbMATH07683016}) and Kulikov-Ortega-Cerd\`a-Nicola-Tilli (\cite{KOCNT}) independently proved similar results in a more general setting. In particular, they proved  concentration estimates and sharp Wehrl-type bounds in two dimensions for ambient spaces with a constant curvature metric. Kalaj (\cite{kalaj2024}) constructed a family of Bergman weights such that the isoperimetric inequality for the hyperbolic metric in several dimensions can be used in the analysis of the appropriate super-level sets. In \cite{kalaj2025faber}, Kalaj and Ramos proved the concentration estimates for these function spaces.

Quantitative results on this topic appeared soon after. The stability of concentration estimates for the short-time Fourier transform was proved in \cite{GomesAndreRamosTilli}, while the appropriate Bergman space counterpart was the subject of the paper \cite{GomezKalajMelentijevicRamos2024}. Frank, Nicola and Tilli proved the quantitative estimates for generalized Wehrl entropies for the short-time Fourier transform \cite{arXiv:2307.14089}. Melentijevi\'{c}(\cite{melentijevic2025}) gave a proof of the analogous estimates for the weighted Bergman spaces (or the wavelet transform) using the hyperbolic adaption of their method. The stability of both localization and general Wehrl-type entropy bounds in the context of spaces of one-dimensional complex polynomials was given in \cite{zbMATH08030975} by Garcia-Ferrero and Ortega-Cerd\`a. Finally, in \cite{arXiv:2412.10940} Nicola, Riccardi and Tilli gave a breakthrough solution to the problem of the uniqueness of extremizers for the Wehrl entropy bounds for symmetric $SU(N)$ coherent states and formulated two general principles for extremizers.

In this paper, motivated by the results of \cite{kalaj2024} and \cite{kalaj2025faber}, we will consider certain Bergman-type classes of functions with log-subharmonic modulus on $\mathbb{R}^n$, $n \geq 2$, with respect to the spherical metric. For $n=2$, these classes contain the corresponding polynomial classes from \cite{zbMATH07683016} and \cite{zbMATH08030975}, as well as nonholomorphic functions. We also treat the higherdimension spherical setting $n>2$. For such spaces, problems of the quantitative concentration and generalized Wehrl-type entropy estimates will be discussed.

Let \(\mathbb{S}^{n}=\{ \xi \in \mathbb{R}^{n+1}\mid \left\lVert \xi \right\rVert _2 = 1\}\) and let \(\dot{\mathbb{R}}^{n}=\mathbb{R}^{n}\cup \{\infty\}\) with the one-point compactification topology. The inverse stereographic projection \(S: \dot{\mathbb{R}}^{n}\to \mathbb{S}^{n}\) is given by
\[
  \xi_k=\frac{2x_k}{1+|x|^2} \text{ for } k = 1,\ldots,n \text{ and } \xi_{n+1}=\frac{1-|x|^2}{1+|x|^2},
\]
with $S(\infty)=-e_{n+1}$.
The inverse of \(S\) is given by
\[
  x_k= \frac{\xi_k}{1+\xi_{n+1}}, k=1,\ldots,n,
\]
for $\xi \neq -e_{n+1},$ with $S^{-1}(-e_{n+1})=\infty$.
In the rest of the paper we use \(x,y\) for the points in \(\mathbb{R}^{n},\) and \(\xi, \eta\) for the corresponding points (under stereographic projection) on the unit sphere \(\mathbb{S}^{n}.\)

The unit sphere \(\mathbb{S}^{n}\) is a hypersurface of \(\mathbb{R}^{n+1},\) hence, \(\mathbb{S}^{n}\) inherits a Riemannian structure from \(\mathbb{R}^{n+1}.\)
The standard metric on \(\mathbb{S}^{n}\) and the volume form are expressed in terms of the stereographic coordinates  as
$$g_{ij}= \left( \frac{2}{1+|x|^2} \right) ^{2}\delta_{ij} \quad \text{and} \quad dm_S(x)= \left( \frac{2}{1+|x|^2} \right) ^{n}\, dx.$$
On every oriented Riemannian manifold the Laplace--Beltrami \footnote{commonly called the Laplace operator} operator is defined by \(\Delta u= \mathrm{div}\left( \mathrm{grad} \,u \right).\) The Laplace operator on the unit sphere \(\mathbb{S}^{n} \) expressed in terms of the stereographic coordinates is given by
\begin{align*}
  \Delta_Su(x) 
  = & \left( \frac{1+\left| x \right| ^2}{2} \right) ^{2}\Delta u + \left( 2-n \right) \frac{1+\left| x \right| ^2}{2}\sum_{k=1}^{n} x_k \frac{\partial u}{\partial x_k}.
\end{align*}
\begin{remark}
  One can show that for \(u : \mathbb{S}^{n}\to \mathbb{R}\) and \(\tilde{u}= u \left( \frac{x}{\left\lVert x \right\rVert } \right) ,x \neq 0\) we have
  \(\Delta_Su= \left.\left( \Delta \tilde{u} \right) \right|_{\mathbb{S}^{n}}.\)
\end{remark}

We define the Bergman weight function \(\mathcal{W}_n : \mathbb{R}^{n} \to (0,\infty)\) as the smooth radial solution of the equation \(\Delta_S  \log \mathcal{W}_n = -1,\) satisfying the condition $\mathcal{W}_n(0)=1$. The equation is imposed on $\mathbb{R}^{n}$ and $\mathcal{W}_n(x)\to 0$,   as $|x|\to \infty$.

\begin{definition}
  For \(0<p<\infty\) and \(\alpha>0\) we say that a complex valued continuous function \(f: \mathbb{R}^{n}\to \mathbb{C}\) belongs to the Bergman-type space \(B_{\alpha,p}\) if
  \begin{equation}\label{eq:definition-of-the-norm}
    \left\lVert f \right\rVert _{\alpha,p}^{p}=\frac{1}{c(\alpha)}\int_{\mathbb{R}^{n}}^{}\left| f(x) \right| ^{p}\mathcal{W}_n^{\alpha}(x)\,  dm_S(x)<\infty,
  \end{equation}
  where 
  \(c(\alpha)= 2^n\int_{\mathbb{R}^{n}}^{}\mathcal{W}_n^{\alpha}(x)\, \frac{dx}{\left( 1+x^2 \right) ^{n}}. \)
\end{definition}
For every \(\xi \in \mathbb{S}^{n}\) there exists an isometry of the unit sphere \(\psi_\xi : \mathbb{S}^{n}\to \mathbb{S}^{n}\) such that \(\psi_\xi(e_{n+1})=\xi\) and \(\psi_\xi \circ \psi_\xi = \mathrm{Id}_{\mathbb{S}^{n}}\) (see \nameref{sec:appendix}). We denote \(\varphi_{x_0}= S^{-1} \circ \psi_{S(x_0)}\circ S.\) Notice that \(\varphi_{x_0 }(0)=x_0\) and \(\varphi_{x_0 } \circ \varphi_{x_0 }=\mathrm{Id}_{\dot{\mathbb{R}}^{n}}.\)
\begin{definition}
  Let \(U \subset \mathbb{R}^{n}\) be an open set. For an upper semicontinuous function \(f : U \to [-\infty,\infty),\) we say that \(f\) is subharmonic with respect to \(\Delta_S\) (or just \(\Delta_S\)-subharmonic) if for every \(x_0 \in U\) there exists \(R >0\), with $\varphi_{x_0 }(\mathbb{B}_R(0))\subset U$ such that
  \begin{equation}\label{eq:the-submean-value-inequality}
    f(x_0 ) \leq   \int_{\mathbb{S}^{n-1}}^{}f\left( \varphi_{x_0 }\left( r \zeta \right)  \right) \, d \sigma_{n-1}(\zeta)
  \end{equation}
  for all \(0<r<R.\) A continuous function \(f: U \to [0,\infty)\) is called  log-subharmonic with respect to \(\Delta_S\) if the function \(\log f\) is \(\Delta_S\)-subharmonic on \(U\setminus f^{-1}(0).\)
\end{definition}
\begin{remark}
  A function \(f \in C^{2}(U)\) is subharmonic iff \(\Delta_S f \geq 0\) which is proved in Section \ref{subsec:mean-value property}.
\end{remark}
\begin{example}
  The function
  \[
    \begin{split}
      F_m (x)=\exp\left\{\frac{4c_m}{n}  \int_{0}^{|x|} \tau  \left(\tau ^2+1\right)^{n-2} \, _2F_1\left(\frac{n}{2},n+m;\frac{n}{2}+1;-\tau ^2\right) \, d\tau \right\} \\
      = \exp\left\{\frac{4c_m}{n}  \int_{0}^{|x|} \tau  \left(\tau ^2+1\right)^{\frac{n}{2}-2} \, _2F_1\left(\frac{n}{2},1-m-\frac{n}{2};\frac{n}{2}+1;\frac{\tau ^2}{\tau ^2+1}\right) \, d\tau \right\}
    \end{split}
  \]
  is a solution of the equation \(\Delta_S \log |F_m|=\frac{c_m }{\left( 1+|x|^2 \right) ^{m}}.\)
  For $n>2$ and \(c_m \geq 0\), we have
  \[
    F_m(x) \leq \exp \left\{\frac{4c_m}{n}\int_{0}^{|x|}\tau \left( 1+\tau ^2 \right) ^{\frac{n}{2}-2} \, d\tau\right\} \leq \exp \left\{4c_m \frac{\left( 1+|x|^2 \right) ^{\frac{n-2}{2}}-1}{n\left( n-2 \right) }\right\}.
  \]

\end{example}
\begin{remark}
  The functions given in the previous example show that the problem we are solving is meaningful for \(n \geq 3.\) Notice that for positive coefficients those functions are mutually different since
  \[
    \Delta_S \log \left( F_1 F_2\dots F_m \right) = \sum_{k=1}^{m} \frac{c_k}{(1+|x|^2)^{k}}.
  \]
  One can see that \(F_1 F_2 \cdot \ldots \cdot F_m \in B_{\alpha,p}\) as long as \(p\left( c_1 +c_2 +\dots +c_m \right)  <   \alpha \frac{\Gamma\left( \frac{n}{2} \right) \Gamma\left( \frac{n}{2}+1 \right) }{\Gamma\left( n \right) }.\)
  Also, in that case we have  \(\limsup_{|x| \to \infty} \left| (F_1 F_2 \cdot \ldots \cdot F_m)(x) \right| ^{p}\mathcal{W}_n^{\alpha}(x)=0.\)
\end{remark}
\begin{remark}\label{rm:ortega-serda}
  For \(n=2\) we have  \(\Delta_S u(x)=\left( \frac{1+|x|^2}{2} \right) ^2 \Delta u,\) therefore, harmonicity and subharmonicity with respect to \(\Delta_S\) are equivalent to standard harmonicity and subharmonicity. Moreover, the Bergman weight function in this case is \(\frac{1}{\left( 1+|z|^2 \right) },\) and hence every holomorphic polynomial of degree \(< \frac{2\alpha +2}{p}\) belongs to \(B_{\alpha,p}.\) Also, the modulus of every holomorphic polynomial is  a \(\Delta_S\)-logsubharmonic function.
\end{remark}
\begin{remark}
  For the sake of brevity, we will use \(dm_S(x)=\frac{2^{n}}{\left( 1+|x|^2 \right) ^{n}}dx\) and \(m_S(A)=\int_{A}^{}\, dm_S(x). \)
\end{remark}
\begin{definition}
  Denote by \(\mathcal{B}_{\alpha,p} \)  the set of real-analytic functions \(f \in B_{\alpha,p},\)  such that   \(|f|\) is \(\Delta_S\)-log-subharmonic and  \(\limsup_{|x| \to \infty} \left| f(x) \right| ^{p}\mathcal{W}_n^{\alpha}(x)<\infty.\)
\end{definition}
Now we can formulate one of our main theorems. Our goal is to maximize the functional:
\[
  C_\Omega^{\alpha} \left( f \right) = \frac{\frac{1}{c(\alpha)}\int_{\Omega}^{}|f(x)|^{p}\mathcal{W}_n^{\alpha}(x)\, dm_S(x) }{\left\lVert f \right\rVert _{\alpha,p}^{p}}
\]
for all measurable \(\Omega\subset \mathbb{R}^{n}\) such that \(m_S(\Omega)=s\) and \(f \in \mathcal{B}_{\alpha,p},\)$ f\neq 0$. If it is evident what $\alpha$ is, we will simply write $C_\Omega(f)$.
\begin{theorem}\label{tm:the-faber-krahn-inequalities-theorem}
  Let \(\alpha >0\) and \(p  >0 \) be fixed. For every \(f \in \mathcal{B}_{\alpha,p}, \) $f \neq0$ and for every measurable $\Omega\subset \mathbb{R}^{n}$ such that $0<m_S(\Omega)<m_S(\mathbb{R}^{n})$ the concentration operator satisfies the inequality
  \[
    C_{\Omega}^{\alpha}(f) \leq  C_{\mathbb{B}}^{\alpha}\left( \mathbf{1} \right),
  \]
  where \(\mathbb{B}\) is the ball centered at the origin such that \(m_S(\mathbb{B})=m_S(\Omega).\)
\end{theorem}

\begin{theorem}\label{tm:convex-functional-theorem}
  Let \(p >0\) and \(\alpha  >0\) and let \(G: [0,\infty) \to \mathbb{R}\) be a convex function.
  Then the maximum value of
  \[
    \int_{\mathbb{R}^{n}}^{}G\left( \left| f(x) \right| ^{p}\mathcal{W}_n^{\alpha}(x) \right) \, dm_S(x)
  \]
  is achieved for \(f(x)\equiv 1,\) subject to the condition that \(f \in \mathcal{B}_{\alpha,p}\) and \(\left\lVert f \right\rVert _{\alpha,p}=1.\)
\end{theorem}
\begin{corollary}\label{popravka1}
  For all \(0<p<q<\infty\) and \(0< \alpha < \beta<\infty\) with \(\frac{p}{\alpha}=\frac{q}{\beta}\) we have
  \[
    \left\lVert f \right\rVert _{\beta,q}\leq \left\lVert f \right\rVert _{\alpha,p}
  \]
  for every \(f \in \mathcal{B}_{\alpha,p}.\)
\end{corollary}
\begin{proof}
  We get the corollary taking \(G(t)=t ^{s},s >1\) in Theorem \ref{tm:convex-functional-theorem}.
\end{proof}
In the next two theorems we state the stability versions of the previous theorems. For $0<m_S(\Omega)<m_S(\mathbb{R}^{n})$, let us denote
\[
  \delta \left( f; \Omega,\alpha \right) = 1-\frac{C_{\Omega}^{\alpha}(f)}{C_{\mathbb{B}}^{\alpha}(\mathbf{1})},
\]
where \(\mathbb{B}\) is the ball centered at the origin such that \(m_S(\mathbb{B})=m_S(\Omega)\). For the following theorem, take $p=2$ and define  \(I_{x_0}(x)= \frac{\left( \mathcal{W}_n \circ \varphi_{x_0 } \right) ^{\alpha/2}(x)}{\mathcal{W}_n^{\alpha/2}(x)}\) extended by zero at its finite pole. Notice that \(C_{\mathbb{B}}(\mathbf{1})=C_{\varphi_{x_0 }(\mathbb{B})}(I_{x_0 }).\) Also, notice that the function $I_{x_0 }$ need not belong to $\mathcal{B}_{\alpha,2}$ since real analyticity may fail at the finite pole of $\varphi_{x_0}$.  Still, the value $\left\lVert I_{x_0 } \right\rVert _{\alpha,2}$ is defined by the formula \eqref{eq:definition-of-the-norm}. We use these functions as comparison profiles in the stability estimate, without requiring them to be admissible extremizers.
\begin{theorem}\label{tm:stability-of-concentration-operator}
  Let  \(\alpha >0\), $0<m_S(\Omega)<m_S(\mathbb{R}^{n})$ and \(f \in \mathcal{B}_{\alpha,2}\) be such that \(\left\lVert f \right\rVert _{\alpha,2}=1.\) There exists a constant \(C = C\left( m_S(\Omega),n, \alpha \right) \) independent of \(f\) such that
  \[
    \inf_{x_0 \in \mathbb{R}^{n}}\left\lVert |f|-I_{x_0 } \right\rVert_{\alpha,2} \leq  C \delta\left( f; \Omega,\alpha \right) ^{\frac{1}{n+2}}.
  \]
\end{theorem}

\begin{theorem}\label{tm:stability-convex-functional}
  Let \(p >0\) and \(\alpha>0.\) Let \(G:[0,\infty)\to \mathbb{R}\) be a convex function and \(f \in \mathcal{B}_{\alpha,p}\) such that  \(\left\lVert f \right\rVert _{\alpha,p}=1.\) Then
  \[
    \int_{T}^{1}\left( G'(t)-G'_-(T) \right)\rho_0(t) \, dt\leq  \int_{\mathbb{R}^{n}}^{}G\left( \mathcal{W}_n^{\alpha}(x) \right) \, dm_S(x) -\int_{\mathbb{R}^{n}}^{}G\left( |f(x)|^p \mathcal{W}_n^{\alpha}(x) \right) \, dm_S(x) ,
  \]
  where \(T=\sup_{x \in \mathbb{R}^{n}} |f(x)|^{p}\mathcal{W}_n^{\alpha}(x)\) and \(\rho_0(t)=m_S\left( \left\{ \mathcal{W}_n^{\alpha}>t \right\} \right)  .\)
\end{theorem}
In the next section, we will give some preliminary notions and results. The calculation of the Bergman weight is given in the third section, while the monotonicity result for the super-level sets of $|f(x)|^p\mathcal{W}_n^{\alpha}(x),$ for $f\in \mathcal{B}_{\alpha,p}$ is the subject of the fourth section. In sections 5 and 6, we will give proofs of localization estimates and generalized Wehrl-type bounds, while sections 7 and 8 are devoted to establishing stability versions of these theorems. Finally, in section 9 we give a review of the main properties of subharmonic functions with respect to the spherical Laplacian.

\section{Preliminaries}
\subsection{Isometries of the unit sphere}
It is known that for the standard metric on the unit sphere the geodesics are great circles and that the distance function is given by \(d_{\mathbb{S}^{n}}(\xi,\eta)=\arccos \langle \xi,\eta \rangle.\)
For every orthogonal \((n+1)\times (n+1)\) matrix $A$, the restriction of the map \(x \mapsto Ax\) is an isometry on the unit sphere. Now, let \(\psi:\:\mathbb{S}^{n}\to\mathbb{S}^{n}\) be an isometry and \(\Psi: \mathbb{R}^{n+1}\to\mathbb{R}^{n+1}\) be given by
\[
  \Psi(x)= \psi\left( \frac{x}{|x|} \right) \cdot |x|,\quad x \in \mathbb{R}^{n+1}\setminus \{0\}, \quad \text{and} \quad \Psi(0)=0.
\]
It is evident that \(\Psi\) is continuous, and it is easy to show that \(\Psi\) is given by \(\Psi(x)=Ax\) where \(A\) is an orthogonal matrix.
\subsection{Invariance of an orthogonal map and the Laplace--Beltrami operator}
It is known that for every isometry \(\psi : (M,g_M) \to (N,g_N)\) of Riemannian manifolds
\[
  \Delta_{M} \psi^* = \psi^* \Delta_N
\]
holds, where \(\Delta_M\) and \(\Delta_N\) are Laplace--Beltrami operators on \(M\) and \(N\) (see \cite[Page 46]{canzani2013analysis} or \cite[p. 27]{chavel1984eigenvalues}). If we use the previous statement with \((M,g_M)=(N,g_N)\) equal to the unit sphere \(\mathbb{S}^{n}\) with the standard metric and \(\psi(x) = Ax,\) where \(A\) is an orthogonal matrix, then for every \(\xi \in \mathbb{S}^{n}\) we have
\[
  \Delta_S(u \circ \psi)(\xi)=\left( \Delta_S\psi^* u \right) (\xi)= \psi^{*}\left( \Delta_Su \right) (\xi) =\left( \Delta_Su \right) \left( \psi(\xi) \right) .
\]
\subsection{Isoperimetric inequality for the unit sphere \texorpdfstring{\(\mathbb{S}^{n}\)}{ S n}}

We say that  measurable sets \(A_k \subset \mathbb{S}^{n}\) converge to \(A\) with respect to volume if the volume of the symmetric difference \(A_k \triangle A\) tends to zero.
For a measurable set \(A \subset \mathbb{S}^{n}\), the perimeter is given by
\[
  P(A)= \inf \left( \liminf_{i \to \infty} \mathcal{H}_{\mathbb{S}^{n}}^{n-1}\left( \partial M_i \right)  \right)
\]
where \(\inf\) is taken over all sequences of compact smoothly embedded n-dimensional manifolds with boundary \(\partial M_i\) converging to \(A\) with respect to volume \cite[14.1.1]{burago1988geometric}.
In the book \cite[Theorem 14.3.1]{burago1988geometric} it is stated that for a measurable set \(A \subset \mathbb{S}^{n}\) the inequality
\[
  P(A) \geq P(D_A)
\]
holds, where \(D_A\) is a ball in \(\mathbb{S}^{n}\) with volume \(V(D_A)=\min\{V(A),V(\mathbb{S}^{n}\setminus A)\}.\)
In paragraph 14.6.1 of the same book it is stated that for a measurable set \(A\) we have \(P(A)\leq \mathcal{H}_{\mathbb{S}^{n}}^{n-1}(\partial A),\) and it is easy to see from the same paragraph that \(\mathcal{H}_{\mathbb{S}^{n}}^{n-1}(\partial D_A)= P( D_A).\) Therefore, for every measurable set \(A \subset \mathbb{S}^{n}\) we have
\[
  \mathcal{H}_{\mathbb{S}^{n}}^{n-1}\left( \partial A \right) \geq \mathcal{H}_{\mathbb{S}^{n}}^{n-1}\left( \partial D_A \right) .
\]
Let \(A \subset \mathbb{S}^{n}\) have a smooth boundary.  Then \(\partial A\) has an induced Riemaniann metric, which makes it a metric space. Therefore, Hausdorff measures \(\mathcal{H}_{\partial A}^k\) are defined correctly. Also, \(\mathbb{S}^{n}\) is a metric space, and hence, it has its own Hausdorff measures \(\mathcal{H}^{k}.\)
In the following calculation, $dV_S$ denotes the Riemannian volume form for the orientation making $S$ orientation preserving, and $\partial A$ has the induced boundary orientation. By \cite[Prop. 12.6 and 12.7]{taylorMeasureTheory2006} we have
\begin{align}\label{eq:hausdorff-measure-the-first-part}
  \begin{split}
    \mathcal{H}_{\mathbb{S}^{n}}^{n-1}\left(\partial  A \right) = & \mathcal{H}^{n-1}_{\partial A}(\partial A)=Vol(\partial A)
    =                                                              \int_{\partial A}^{}i_N dV_S                                \\
    =                                                             & \int_{S^{-1}\left( \partial A \right) }^{}S^{*}i_N dV_S
    =                                                              \int_{S^{-1}(\partial A)}^{}i_{S_*^{-1}N}S^{*}dV_S
  \end{split}
\end{align}
where \(N\) is the outward unit normal vector on \(\partial A.\) Notice that since \(N\) is the unit vector we see that the vector \(S^{-1}_* N=\sum_{k=1}^{n} N_k \frac{\partial }{\partial x_k} \) satisfies \(\sum_{k=1}^{n} 4\frac{N_k^2}{\left( 1+x^2 \right)^2 }= 1. \) Then \(S^{-1}_* N\) is a normal vector on \(S^{-1}(\partial A)\) but is not necessarily a Euclidean unit vector. The Euclidean unit normal vector to \(S^{-1}\left( \partial A \right) \) is \(\tilde{N} = \frac{2}{1+|x|^2} S^{-1}_*N.\) Since \(S^{*}dV_S=\frac{2^{n}}{\left( 1+x^2 \right) ^{n}}dx\) from (\ref{eq:hausdorff-measure-the-first-part}) we obtain
\begin{equation}\label{eq:hausdorff-measure-of-the-boundary-of-set}
  \begin{split}
    \mathcal{H}^{n-1}_{\mathbb{S}^{n}}(\partial A)=\int_{S^{-1}(\partial A)}^{}\frac{2^{n-1}}{\left( 1+|x|^2 \right) ^{n-1}}i_{\tilde{N}} dV_{\mathbb{R}^{n}}
     & = \int_{S^{-1}(\partial A)}^{}\frac{2^{n-1}}{\left( 1+|x|^2 \right) ^{n-1}}\, d \mathcal{H}_{S^{-1}\left( \partial A \right) }^{n-1}(x) \\
     & = \int_{S^{-1}(\partial A)}^{}\frac{2^{n-1}}{\left( 1+|x|^2 \right) ^{n-1}}\, d \mathcal{H}_{\mathbb{R}^{n}}^{n-1}(x).
  \end{split}
\end{equation}
For every \(0<r<\pi\) we have
\[
  V(r)=Vol_{\mathbb{S}^{n}} \{ \xi \in \mathbb{S}^{n} \mid d_S(\xi,e_{n+1})<r\} = m_S \left( \left\{  |x| <\tan \frac{r}{2} \right\}  \right)
\]
and
\[
  P(r)=P\left( \left\{ \xi \in \mathbb{S}^{n} \mid d_S(\xi,e_{n+1}) <r\right\}  \right)=\mathcal{H}_{\mathbb{S}^{n}}^{n-1}\left( \left\{  \xi \in \mathbb{S}^{n}\mid d_S(\xi,e_{n+1})=r  \right\}  \right) .
\]
Since the function \(r \mapsto V(r)\) is increasing, it has an inverse function \(R(v)=r.\) We define the function
\[
  \Theta(v)=\frac{v}{P^2\left( R(v) \right) },
\]
for $0<v<Vol(\mathbb{S}^{n})$.
Notice that for every measurable \(A \subset \mathbb{S}^{n}\) with $0<Vol(A)<Vol(\mathbb{S}^{n})$ we have
\begin{equation}\label{eq:isoperimetric-inequality-for-the-monotonicity-theorem}
  \left( \mathcal{H}_{\mathbb{S}^{n}}^{n-1}(\partial A) \right) ^2\geq \frac{Vol(A)}{\Theta(Vol(A))}.
\end{equation}

\section{Bergman Weight Function}\label{sec:bergman-weight-function}
The Bergman weight function on the unit disk in \(\mathbb{C}\) is given by \(w(z)=1-|z|^2,\) it is radial and solves the equation \(\Delta_h \log w=-1.\) Hence, we are searching for a Bergman weight function on $\mathbb{R}^{n}$ as a solution of the equation:
\[
  \Delta_S \log (u)= -1,
\]
where \(u(x)=g(|x|)\). If we denote  \(h(r)=\log g(r)\) then
\begin{align*}
  \Delta_S \log (u) 
   & = \frac{1+r^2}{4}\left[ (1+r^2) h''(r)+h'(r)\frac{(1+r^2)(n-1)+2(2-n)r^2}{r}\right].
\end{align*}
Taking \(k(r)=h'(r)\) we get the equation
\begin{equation}\label{eq:bergman-weight-function-function-k}
  \frac{1+r^2}{4}\left[ \left( 1+r^2 \right) k'+k\frac{(1+r^2)(n-1)+2(2-n)r^2}{r} \right]= -1
\end{equation}
and conclude that
\(k(r)=\frac{r \left(r^2+1\right)^{n-2} \left(-4 \, _2F_1\left(\frac{n}{2},n;\frac{n}{2}+1;-r^2\right)+C_1 n r^{-n}\right)}{n},\)
where
\[
  _2F_1\left( a,b;c;t \right) = \sum_{j=0}^{\infty} \frac{(a)_j(b)_j}{(c)_j j!}t^j, \quad \text{for } |t|<1
\]
is the hypergeometric function.
We want our solution to be defined at \(r=0,\) and hence for \(C_1= 0\) we have
\begin{equation}\label{eq:the-formula-for-k_c}
  k(r)=-\frac{4}{n}r \left(r^2+1\right)^{n-2}  \, _2F_1\left(\frac{n}{2},n;\frac{n}{2}+1;-r^2\right) = \frac{-4(1+r^2)^{n-2}}{r^{n-1}} \int_{0}^{r}\frac{\tau^{n-1}}{(1+\tau^2)^n}\, d\tau .
\end{equation}
Notice that the hypergeometric function \(_2F_1(a,b;c;z)\), for parameters used here, is well-defined and analytic on \(\mathbb{C} \setminus [1,\infty),\) therefore, the function \(k\) is well-defined and a solution of equation (\ref{eq:bergman-weight-function-function-k}).
From the Pfaff identity \(_2F_1(a,b;c;z)=(1-z)^{-b} \,_2F_1(c-a,b;c;\frac{z}{z-1})\) we get
\[
  k(r)=-\frac{4r}{n}\left( 1+r^2 \right) ^\frac{n-4}{2} \, _2 F_1 \left( \frac{n}{2},1-\frac{n}{2};1+\frac{n}{2};\frac{r^2}{1+r^2} \right).
\]
Notice that
\begin{align*}
  _2F_1'\left( \frac{n}{2},1-\frac{n}{2};1+\frac{n}{2};t \right) & = \frac{n(2-n)}{2(n+2)}{}_2F_1\left( \frac{n}{2}+1,2-\frac{n}{2};2+\frac{n}{2},t \right)                                          \\
                                                                 & =  \frac{n(2-n)}{2(n+2)} \left( 1-t \right) ^{\frac{n}{2}-1} {}_2F_1\left( 1,n;2+\frac{n}{2};t \right) \leq 0, \quad t \in (0,1),
\end{align*}
hence we have
\
\[
  -\frac{4r}{n}\left( 1+r^2 \right) ^{\frac{n-4}{2}} \leq k\left( r \right) \leq  -\frac{4r}{n}\left( 1+r^2 \right) ^{\frac{n-4}{2}}\frac{\Gamma\left( 1+ \frac{n}{2} \right) \Gamma\left( \frac{n}{2} \right) }{\Gamma(n)}
\]
Since \(h(r)=\int_{0}^{r}k(t)\, dt\) and
\[
  \int_{0}^{r}-\frac{4t}{n}\left( 1+t^2 \right) ^{\frac{n-4}{2}}  \, dt=  -4  \frac{\left( 1+r^2 \right) ^{\frac{n-2}{2}}-1}{n(n-2)},
\]
we obtain
\[
  \exp\{-4  \frac{\left( 1+r^2 \right) ^{\frac{n-2}{2}}-1}{n(n-2)}\}\leq u(x)\leq \exp\{-4\frac{\Gamma\left( 1+ \frac{n}{2} \right) \Gamma\left( \frac{n}{2} \right) }{\Gamma(n)} \cdot \frac{\left( 1+r^2 \right) ^{\frac{n-2}{2}}-1}{n(n-2)}\},
\]
for \(n >2.\) For \(n=2\) we have  \(k(r)=-2r\left( 1+r^2 \right) ^{-1},\) hence \(h(r)=-\log \left( 1+r^2 \right).\) Consequently, \(u(r)=\frac{1}{ 1+r^2 }.\)

We denote \(\mathcal{W}_n(x)=u(x)=g(|x|)=e^{\int_{0}^{|x|}k(\tau)\, d\tau }.\) Note that $\mathcal{W}_n(x)$ is a radial function and sometimes we will write $\mathcal{W}_n(r)$ for the value of $\mathcal{W}_n$ at points with radius $r$.


\section{Monotonicity Theorem}
In this section we prove the monotonicity theorem, from which the theorems given in the next sections follow.
\begin{theorem}\label{tm:the-monotonicity-theorem}
  Let \(a \geq 0,\) and \(\alpha>0.\)  Assume that \(f\) is a real-analytic complex-valued function such that \(|f|: \mathbb{R}^{n}\to [0,\infty)\) is a log-subharmonic function with respect to \(\Delta_S.\)  Assume that the function \(u(x)=|f(x)|^{a}\mathcal{W}_n^{\alpha}(x)\) is a non-constant function such that \(\lim_{|x| \to \infty}u(x)=0\) uniformly. Then \(\rho(t)= m_S\left( \left\{ x \in \mathbb{R}^{n}\mid u(x)>t \right\}  \right) \) is absolutely continuous on \([0,\max u]\) and
  \begin{equation}\label{eq:the-monotonicity-theorem}
    \alpha  \Theta\left( \rho(t) \right) \rho'(t)+\frac{1}{t}\leq 0
  \end{equation}
  for almost every \(t \in (0,\max u)\)
\end{theorem}
\begin{proof}
  It suffices to consider $a>0$, since the case $a=0$ gives $u=\mathcal W_n^\alpha$ and is therefore covered by taking $f\equiv1$.

  Let \(A_t=\{x\in\mathbb R^n\mid u(x)>t\}\). Then \(\rho(t)=m_S(A_t)\). Notice that \(u\) is real analytic on \(\{u>0\}\) and cannot be constant on any nonempty open subset of this set. By the zero-set theorem for real-analytic functions \cite{Krantz2002Parks}, every positive level set \(\{u=t\}\) has measure zero, which implies that \(\rho\) is continuous on \((0,\max u)\). The same theorem implies that the critical set of \(u\) in \(\{u>0\}\) has measure zero. By Sard’s theorem, almost every \(t\in(0,\max u)\) is a regular value. For every such \(t\), we have \(\partial A_t=\{u=t\}\), and \(\partial A_t\) is a smooth hypersurface.

  We start with the coarea formula (see \cite{Brothers1988})
  \[
    \rho(t)= \int_{A_t}^{}\frac{2^{n}}{\left( 1+|x|^2 \right) ^{n}}\, dx =\int_{t}^{\max u}\int_{|u(x)|=\kappa}^{}\frac{2^{n}\left| \nabla u \right| ^{-1}}{\left( 1+|x|^2 \right) ^{n}}\, d \mathcal{H}^{n-1}_{\mathbb{R}^{n}}(x) \, d \kappa.
  \]
  Letting $t\to0$ in the coarea formula and using monotone convergence, we see that its nonnegative level-set integrand is integrable on $(0,\max u)$, since $\rho(t)\to\rho(0)<\infty$. Consequently, $\rho$ is absolutely continuous on $[0,\max u]$. Consequently, we have
  \begin{equation}\label{eq:derivative-of-function-rho}
    - \rho'(t)=\int_{u=t}^{}\left| \nabla u \right| ^{-1}\frac{2^{n}\,  d \mathcal{H}^{n-1}_{\mathbb{R}^{n}}(x)}{\left( 1+|x|^2 \right) ^{n}},
  \end{equation}
  for almost every $t \in (0,\max u)$.
  Using the same approach as in the papers \cite{NicolaTilli2022,kulikov2022functionals-with-extrema,kalaj2024}, our next step is to apply the Cauchy-Schwarz inequality to get
  \begin{align}\label{eq:monotonicity-theorem-after-applying-Cauchy-Schwarz}
    \begin{split}
      \left[ \mathcal{H}^{n-1}_{\mathbb{S}^{n}}\left( S\left( \partial A_t \right)  \right)   \right] ^2
       & = \left[ \int_{\partial A_t}^{}\frac{2^{n-1} }{\left( 1+|x|^2 \right) ^{n-1}}\, d \mathcal{H}^{n-1}_{\mathbb{R}^{n}}  \right] ^2                                                                                                                                                           \\
       & \leq  \int_{\partial A_t}^{}\left| \nabla u \right| ^{-1} \frac{2^{n}\, d \mathcal{H}^{n-1}_{\mathbb{R}^{n}}(x)}{\left( 1+|x|^2 \right) ^{n}} \cdot \int_{\partial A_t}^{}\frac{\left| \nabla u \right| 2^{n-2}}{\left( 1+|x|^2 \right) ^{n-2}}\, d \mathcal{H}^{n-1}_{\mathbb{R}^{n}}(x),
    \end{split}
  \end{align}
  for almost all \(t \in (0,\max u).\)
  Let us denote by \(\nu = \nu(x)\) the outward unit normal to \(\partial A_t\) at the point \(x.\) Notice that \(\nabla u\) is parallel to \(\nu,\) but directed in the opposite direction. Therefore, we have  \(\left| \nabla u \right| = - \langle \nabla u,\nu \rangle.\) Also, for every \(x \in \partial A_t\) we have \(u(x)=t\); hence
  \[
    \frac{\left| \nabla u(x) \right| }{t}=\frac{\left| \nabla u \right| }{u}=-\langle \nabla \log u(x),\nu \rangle.
  \]
  The second integral on the RHS of (\ref{eq:monotonicity-theorem-after-applying-Cauchy-Schwarz}) can be calculated using Gauss's divergence theorem:
  \begin{align*}
    \int_{\partial A_t}^{}\frac{\left| \nabla u \right| \, d \mathcal{H}^{n-1}(x)}{\left( 1+|x|^2 \right) ^{n-2}}= - t \int_{A_t}^{}\mathrm{div}\left( \frac{\nabla \log u(x)}{\left( 1+|x|^2 \right) ^{n-2}} \right) \, dx   = -4t \int_{A_t}^{}\frac{1}{\left( 1+|x|^2 \right) ^{n}}\Delta_S \log u(x)\, dx.
  \end{align*}
  As one can see
  \[
    \Delta_S \log u(x)=a \Delta_S  \log |f(x)|+\alpha \Delta_S \log \mathcal{W}_n(x) \geq -\alpha.
  \]
  Substituting the previous discussion into (\ref{eq:monotonicity-theorem-after-applying-Cauchy-Schwarz}), and using  (\ref{eq:derivative-of-function-rho}), we get
  \begin{align*}
    \left( \mathcal{H}^{n-1}_{\mathbb{S}^{n}}\left( S\left( \partial A_t \right)  \right)   \right) ^2
     & \leq \left( -\rho'(t) \right) \int_{\partial A_t}^{}\frac{\left| \nabla u \right| 2^{n-2} d \mathcal{H}^{n-1}_{\mathbb{R}^{n}}(x)}{\left( 1+|x|^2 \right) ^{n-2}} \\
     & \leq - 2^{n}t \alpha  \rho'(t) \frac{\rho(t)}{2^{n}}                                                                                                              \\
     & =-t \alpha  \rho'(t)\rho(t).
  \end{align*}
  From (\ref{eq:isoperimetric-inequality-for-the-monotonicity-theorem}) we obtain \(\left[ \mathcal{H}^{n-1}_{\mathbb{S}^{n}}\left( S\left( \partial A_t \right)  \right)   \right] ^2  \geq  \frac{\rho(t)}{\Theta(\rho(t))},\)
  hence,
  \[
    \frac{\rho(t)}{\Theta (\rho(t))} \leq -t \alpha  \rho'(t)\rho(t),
  \]
  i.e. \(\alpha  \Theta\left( \rho(t) \right) \rho'(t)+\frac{1}{t}\leq 0,\) for almost every \(t \in (0,\max u).\)
\end{proof}
\begin{remark}\label{rm:the-function-g-is-constant-when-f-is-equal-1}
  Notice that for \(f \equiv 1,\) all inequalities in the previous proof become equalities for all values of \(a\) and \(\alpha.\) Indeed, the isoperimetric inequality becomes an equality since \(A_t\) is a ball centered at \(x=0\) and in the Cauchy-Schwarz inequality we have equality since \(|\nabla u|\) and \(1+|x|^2\) are constant on \(\partial A_t.\) Therefore, for \(f\equiv 1\) and \(\rho_0(t)=m_S\left( \left\{ \mathcal{W}_n^{\alpha}>t \right\}  \right) \) we have
  \[
    \alpha \Theta(\rho_0(t))\rho_0'(t)+\frac{1}{t}=0.
  \]
\end{remark}
\begin{remark}\label{rm:inverse-functions-differentiability-identities}
  For a function \(u\) we define the decreasing rearrangement of \(u\) as
  \[
    u^*(s)=\sup \{ t \geq 0 : \rho(t) > s \}
  \]
  (more details on rearrangements can be found in \cite{Baernstein}).
  Since \(\rho\) is a strictly decreasing function, we can see that  \( \left. u^* \right|_{[0,m_S(\mathbb{R}^{n})]}   =\left.\rho\right|_{[0,\max u]}^{-1}.\) By the same argumentation as in \cite[Lemma 3.2]{NicolaTilli2022} we conclude that \(u^*\) is absolutely continuous on \([0,m_S(\mathbb{R}^{n})].\) By elementary calculus and \eqref{eq:the-monotonicity-theorem}  we see that
  \[
    \alpha \Theta\left( s \right) u^{*}(s) + (u^{*})'(s) \geq  0,
  \]
  for almost every \(s \in (0,m_S(\mathbb{R}^{n}).\) Also, for \(v=\mathcal{W}_n^\alpha,\) we obtain
  \(
  \alpha \Theta\left( s \right) v^{*}(s) + (v^{*})'(s)  = 0.
  \)
\end{remark}
\section{Proof of Theorem \ref{tm:the-faber-krahn-inequalities-theorem}}
First, we will deal with functions such that \(\lim_{|x| \to \infty} \left| f(x) \right| ^{p}\mathcal{W}_n^{\alpha}=0.\) After the next lemma, we give a proof of Theorem \ref{tm:the-faber-krahn-inequalities-theorem} in the general case.
\begin{lemma}\label{lm:the-faber-krahn-inequalities-lemma}
  Let \(\alpha >0\) and \(p  >0 \) be fixed and let \(f \in \mathcal{B}_{\alpha,p},f \neq 0\), be such that
  \[
    \lim_{|x| \to \infty} |f(x)|^{p}\mathcal{W}_n^{\alpha}(x) = 0.
  \]
  Assume that $0<m_S(\Omega)<m_S(\mathbb{R}^{n})$. The concentration operator satisfies the inequality
  \[
    C_{\Omega}(f) \leq  C_{\mathbb{B}}\left( \mathbf{1} \right),
  \]
  where \(\mathbb{B}\) is the ball centered at the origin such that \(m_S(\mathbb{B})=m_S(\Omega).\)
\end{lemma}
\begin{proof} As usual, denote \(u=|f|^{p}\mathcal{W}_n^{\alpha}.\) Without loss of generality, we can assume that \(\left\lVert f \right\rVert _{\alpha,p}=1.\) By the bathtub lemma we have
  \[
    \int_{\Omega}^{}u(x)\, dm_S(x) \leq \int_{A_t}^{}u(x)\, dm_S(x) ,
  \]
  where  \(t\) is chosen so that \(m_S(A_t)=m_S(\Omega).\)
  Following the idea of \cite{melentijevic2025} we have the identity
  \begin{align*}
    \int_{A_t}^{}u(x)\, dm_S(x) & = \int_{0}^{t}\left( \int_{\{u>t\}}^{}\, dm_S(x) \right)  \, d\tau + \int_{t}^{T}\left( \int_{\{u>\tau\}}^{}\, dm_S(x)  \right) \, d\tau \\
                                & = t \rho(t)+\int_{t}^{T}\rho(\tau)\, d\tau,
  \end{align*}
  where \(T=\max u\leq 1\), by Theorem \ref{tm:point-evaluation-theorem}. Denote \(\rho_0(t)=m_S\left( \{ \mathcal{W}_n^{\alpha}>t\} \right) .\) Our goal is to prove
  \[
    \varPsi(t)= t \rho(t)+\int_{t}^{T}\rho(\tau)\, d\tau - \beta (t) \rho_0(\beta(t)) - \int_{\beta(t)}^{1}\rho_0(\tau)\, d\tau \leq  0,
  \]
  where \(\beta(t)\) is such that \( \rho_0\left( \beta(t) \right)  = \rho(t).\)
  Since
  \[
    \varPsi'(t)= t \rho'(t) - \beta(t) \rho_0'\left( \beta(t) \right)   \beta'(t),
  \]
  and since \(\rho(t)=\rho_0(\beta(t)) \Rightarrow \rho'(t)=\rho_0'(\beta(t))\beta'(t),\) we have
  \[
    \varPsi'(t) = \rho'(t) \left( t-\beta(t) \right) ,
  \]
  for almost every $t \in (0,T)$.
  Notice that \(\rho(t)\) is a decreasing function on \([0,T]\), and hence \(\rho'(t) <0\) for almost every $t \in (0,T)$ (see Theorem \ref{tm:the-monotonicity-theorem}). Also, since \(t > \beta(t) \Leftrightarrow \rho_0(t)<\rho_0\left( \beta(t) \right) = \rho(t)\) we have that the sign of \(\varPsi'\) is opposite to the sign of \(\rho(t)-\rho_0(t)\) almost everywhere. From Theorem \ref{tm:the-monotonicity-theorem}, we conclude that
  \[
    \alpha  \Theta\left( \rho(t) \right) \rho'(t) + \frac{1}{t}\leq  0 \quad \text{and} \quad \alpha \Theta \left( \rho_0(t) \right) \rho_0'(t)+\frac{1}{t}=0,
  \]
  therefore,
  \[
    \alpha \left( \Theta (\rho(t)) \rho'(t)-\Theta(\rho_0(t)) \rho_0'(t)\right) \leq  0,
  \]
  for almost every \(t \in (0,T).\)
  In light of this, we see that the function \(V(\rho(t))-V(\rho_0(t))\) is decreasing, where \(V(x)=\alpha \int_{m_S(\mathbb{R}^{n})/2}^{x}\Theta(\tau)\, d\tau .\)
  Indeed, \(V\) is a locally absolutely continuous function and \(\rho\) is a decreasing absolutely continuous function, therefore \(V(\rho)\) is  also locally absolutely continuous, as well as \(V(\rho)-V(\rho_0).\) Hence, a non-positive derivative of \(V(\rho)-V(\rho_0)\) implies that this function is decreasing.
  Since \(V\) is an increasing function, we deduce that \(\rho(t)-\rho_0(t)\) is nonnegative for some initial interval, and nonpositive thereafter. Therefore, \(\varPsi'\) takes negative and then positive values, which means that the function \(\varPsi\) first decreases and then increases. Since \(\max \{ \varPsi(0),\varPsi(T)\}= 0\) we have  \(\varPsi(t)\leq 0.\) Therefore,
  \[
    \int_{A_t}^{}u(x)\, dm_S(x) \leq \int_{\{\mathcal{W}_n^{\alpha}>\beta(t)\}}^{}\mathcal{W}_n^{\alpha}(x)\, dm_S(x).
  \]
  Since \(\mathcal{W}_n^{\alpha}\) is a radial and strictly decreasing in the radius, we have the desired conclusion.
\end{proof}

\paragraph{\textbf{Proof of Theorem \ref{tm:the-faber-krahn-inequalities-theorem}}}
In Lemma  \ref{lm:the-faber-krahn-inequalities-lemma} it is proven that for every \(\Omega\) and \(f\) satisfying its hypotheses, the inequality
\[
  \frac{1}{c(\alpha)\left\lVert f \right\rVert _{\alpha,p}^{p}} \int_{\Omega}|f(x)|^p \mathcal{W}_n^{\alpha}(x) \, dm_S(x)\leq \frac{1}{c(\alpha)} \int_{\mathbb{B}}^{}\mathcal{W}_n^{\alpha}(x)\, dm_S(x),
\]
holds.
Notice that \(f \in B_{\alpha,p}\) implies \(f \in B_{\alpha+\epsilon,p},\) and that \(\limsup_{|x| \to \infty} \left| f(x) \right| ^{p}\mathcal{W}_n^{\alpha}(x)<\infty\) implies \(\limsup_{|x| \to \infty} \left| f(x) \right| ^{p}\mathcal{W}_n^{\alpha+\epsilon}(x)=0,\) for every \(\epsilon >0.\) Hence,
\[
  \frac{1}{c(\alpha+\epsilon)\left\lVert f \right\rVert _{\alpha+\epsilon,p}^{p}}\int_{\Omega}^{}\left| f(x) \right| ^{p}\mathcal{W}_n^{\alpha+\epsilon}(x)\, dm_S(x) \leq \frac{1}{c(\alpha+\epsilon)}\int_{\mathbb{B}}^{}\mathcal{W}_n^{\alpha+\epsilon}(x)\, dm_S(x).
\]
By the monotone convergence theorem we have \(\lim_{\epsilon \to 0+} c(\alpha+\epsilon) =c(\alpha),\) \(\lim_{\epsilon \to 0+}\left\lVert f \right\rVert _{\alpha+\epsilon,p}^{p}=\left\lVert f \right\rVert _{\alpha,p}^{p} \) and \(\lim_{\epsilon \to 0+} \int_{\Omega}^{}|f(x)|^p \mathcal{W}_n^{\alpha+\epsilon}(x)\, dm_S(x) =\int_{\Omega}^{}|f(x)|^p\mathcal{W}_n^{\alpha}\, dm_S(x). \) Therefore, we have
\[
  \frac{1}{c(\alpha)\left\lVert f \right\rVert _{\alpha,p}^p}\int_{\Omega}^{}\left| f(x) \right| ^{p}\mathcal{W}_n^{\alpha}(x)\, dm_S(x) \leq \frac{1}{c(\alpha)}\int_{\mathbb{B}}^{}\mathcal{W}_n^{\alpha}(x)\, dm_S(x). \eqno\qed
\]
\section{Proof of Theorem \ref{tm:convex-functional-theorem}}

\noindent\textbf{Proof of Theorem \ref{tm:convex-functional-theorem}.}
Let \(u=|f|^{p}\mathcal{W}_n^\alpha,\) and \(A_t=\{u > t \}.\) Here we use the standard notation of measure theory, namely \(s_+=\max \{s,0\}\) for $s \in \mathbb{R}$.
By Theorem \ref{tm:the-faber-krahn-inequalities-theorem}, we have
\begin{gather*}
  \int_{\mathbb{R}^{n}}^{}\left( u(x) - t \right) _+\, dm_S(x)\\
  = \int_{A_t}^{}\left( u(x)-t \right) \, dm_S = \int_{A_t}^{}u\, dm_S  - t m_S(A_t) \leq \int_{A_t^*}^{}\mathcal{W}_n^{\alpha}(x)\, dm_S - t m_S(A_t^*)  \\
  = \int_{A_t^*}^{}\left( \mathcal{W}_n^{\alpha}(x)-t \right) \, dm_S \leq \int_{A_t^*}^{}\left( \mathcal{W}_n^{\alpha}(x)-t \right) _+\, dm_S \leq \int_{\mathbb{R}^{n}}^{}\left( \mathcal{W}_n^{\alpha}(x)-t \right)_+ \, dm_S
\end{gather*}
where \(A^*_t\) is the ball centered at \(0 \in \mathbb{R}^{n}\) with spherical measure equal to \(m_S(A_t).\)
For every convex function \(G\) such that \(G'(0)>-\infty\)
we have
\[
  G(u)=G(0)+G'(0) u + \int_{0}^{\infty}\left( u-t \right) _+\, dG'(t),
\]
therefore,
\begin{gather*}
  \int_{\mathbb{R}^{n}}^{}G\left( u(x) \right) \, dm_S(x)\\
  =G(0) m_S(\mathbb{R}^n)+ \int_{\mathbb{R}^{n}}^{}u(x)\, dm_S(x) \cdot G'(0) + \int_{0}^{\infty}\left( \int_{\mathbb{R}^{n}}^{}\left( u(x)-t \right)_+ \, dm_S(x)  \right) \, dG'(t)\\
  \leq G(0) m_S(\mathbb{R}^n)+   \int_{\mathbb{R}^{n}}^{}\mathcal{W}_n^{\alpha}(x)\, dm_S(x) \cdot G'(0)+ \int_{0}^{\infty}\left( \int_{\mathbb{R}^{n}}^{}\left( \mathcal{W}_n^\alpha(x)-t \right) _+\, dm_S(x)  \right) \, dG'(t)\\
  = \int_{\mathbb{R}^{n}}^{}G(\mathcal{W}_n^{\alpha}(x))\, dm_S(x),
\end{gather*}
where we used the fact that a convex function \(G\) has a derivative in almost every point and that \(dG'\) is a positive measure. If \(G'(0)= -\infty\) then, by the above, we have
\[
  \int_{\mathbb{R}^{n}}^{}G_\epsilon\left( u(x) \right) \, dm_S(x) \leq \int_{\mathbb{R}^{n}}^{}G_{\epsilon}\left( \mathcal{W}_n^{\alpha}(x) \right) \, dm_S(x),
\]
for \(G_\epsilon\left( u \right) = \max \{ G(u),G(0)-\frac{u}{\epsilon}\}\),  whose derivative is bounded almost everywhere. Since this family converges monotonically to \(G\) as \(\epsilon \to 0\) we have the proof in this case.\hfill\(\qed\)
\section{Proof of Theorem \ref{tm:stability-of-concentration-operator}}
\begin{proof}
  Let $s_0 = m_S(\Omega)$, $ u=|f|^2 \mathcal{W}_n^\alpha$ and \(v=\mathcal{W}_n^\alpha\). Assume for now that \(\limsup_{|x| \to \infty}u(x)=0.\) Notice that since \(\rho(t) = m_S(\{u > t\})\) we have \(s=m_S(\{u > u^{*}(s)\})\) and \(\int_{0}^{s_0}v^{*}(s)\, ds =\int_{\{v>v^{*}(s_0 )\}}^{}v(x)\, dm_S(x).  \)
  Also, denote \(\delta_{s_0 }= 1- \frac{\int_{0}^{s_0 }u^*(s)\, ds }{\int_{0}^{s_0 }v^*(s)\, ds }.\) \\
  Since \(\int_{\Omega}^{}u(x)\, dm_S(x) \leq \int_{\{u>u^*(s_0 )\}}^{}u(x)\, dm_S(x),  \) we have  \(\delta_{s_0 }\leq \delta\left( f;\Omega,\alpha \right). \)
  Choose \(s^*\in(0,m_S(\mathbb{R}^{n})\) such that \(u^*(s^*)=v^*(s^*)\). Such a point exists by continuity and the equality of the total masses.
  Using the same idea from \cite{GomezKalajMelentijevicRamos2024,GomesAndreRamosTilli} one can show that
  \begin{equation}\label{eq:stability-bound-from-above}
    \int_{0}^{s^*}\left( v^{*}(s)-u^{*}(s) \right) \, ds
    \leq c(\alpha)
    \begin{cases}
      \delta_{s_0 } \cdot \frac{\int_{0}^{s_0}v^{*}(s)\, ds }{c(\alpha)- \int_{0}^{s_0}v^{*}(s)\, ds }, & s_0  \geq s^{*} \\
      \delta_{s_0 }                                                                                     & s_0 <  s^{*}
    \end{cases}
    \leq c(\alpha) \delta_{s_0} \cdot F_{\alpha}(s_0),
  \end{equation}
  where
  \[
    F_\alpha(s_0)= \max\left\{ 1  , \frac{\int_{0}^{s_0}v^{*}(s)\, ds }{c(\alpha)- \int_{0}^{s_0}v^{*}(s)\, ds } \right\}.
  \]
  Indeed, denote \(r(s)=\frac{u^{*}(s)}{v^{*}(s)},\) for \(s \in (0,m_S(\mathbb{R}^{n})).\) Notice that by Remark \ref{rm:inverse-functions-differentiability-identities}
  \[
    r'(s)=\frac{(u^{*})'(s) v^{*}(s)-(v^{*})'(s)u^{*}(s)}{\left( v^*(s) \right) ^2 } \geq \frac{-\alpha \Theta(s)u^{*}(s) v^{*}(s)+\alpha \Theta(s)v^{*}(s)u^{*}(s) }{\left( v^{*}(s)    \right)^2 }= 0,
  \]
  hence, \(r\) is a nondecreasing function. We will consider two cases.
  \paragraph{\textsc{Case 1: \(m_S(\mathbb{R}^{n})>s_0 >s^*\)}} Since \(r(s_0 )\geq r(s^*)= 1\) we have
  \begin{align*}
    \int_{0}^{s_0 }\left( v^*(s)-u^*(s) \right) \, ds & = \int_{s_0 }^{m_S(\mathbb{R}^n)}\left( u^*(s)-v^*(s) \right) \, ds = \int_{s_0 }^{m_S(\mathbb{R}^n)}u^*(s) \left( 1-\frac{1}{r(s)} \right)\, ds \\
                                                      & \geq \left( 1-\frac{1}{r(s_0 )} \right) \int_{s_0 }^{m_S(\mathbb{R}^n)}u^*(s)\, ds.
  \end{align*}
  By the same argument, we have
  \[
    \int_{s^*}^{s_0 }(u^*(s)-v^*(s))\, ds \leq \left( 1-\frac{1}{r(s_0 )} \right) \int_{s^*}^{s_0 }u^*(s)\, ds.
  \]
  Using the previous inequalities we conclude
  \begin{align*}
         & \int_{0}^{s^*}\left( v^*(s)-u^*(s) \right) \, ds  \leq \int_{0}^{s_0 }\left( v^*(s)-u^*(s) \right) \, ds + \int_{s^*}^{s_0 }\left( u^*(s)-v^*(s) \right) \, ds                                                                                                  \\
    \leq & \int_{0}^{s_0 }\left( v^*(s)-u^*(s) \right) \, ds + \int_{s^*}^{s_0 }u^*(s)\, ds \cdot \frac{\int_{0}^{s_0 }\left( v^*(s)-u^*(s) \right) \, ds }{\int_{s_0 }^{m_S(\mathbb{R}^n)}u^*(s)\, ds }                                                                   \\
    \leq & \frac{\int_{0}^{s_0 }\left( v^*(s)-u^*(s) \right) \, ds}{\int_{0}^{s_0 }v^*(s)\, ds } \cdot \left( \int_{0}^{s_0 }v^*(s)\, ds  +\int_{s^*}^{s_0 }u^*(s)\, ds \cdot \frac{\int_{0}^{s_0 }v^*(s)\, ds }{\int_{s_0 }^{m_S(\mathbb{R}^n)}u^*(s)\, ds  } \right)     \\
    \leq & \delta_{s_0 }  \int_{0}^{s_0 }v^*(s)\, ds \cdot \frac{\int_{s^*}^{m_S(\mathbb{R}^n)}u^*(s)\, ds }{\int_{s_0 }^{m_S(\mathbb{R}^n)}u^*(s)\, ds } \leq \delta_{s_0 }\int_{0}^{s_0 }v^*(s)\, ds \cdot \frac{c(\alpha)}{\int_{s_0 }^{m_S(\mathbb{R}^n)}v^*(s)\, ds},
  \end{align*}
  where the last inequality follows since \(u^*(s)\geq v^*(s)\) for \(s \geq s^*\), and $\int_{s^*}^{m_S(\mathbb{R}^{n})}u^*(s)\, ds \leq \int_{0}^{m_S(\mathbb{R}^{n})}u^*(s)\, ds = c(\alpha).$
  \paragraph{\textsc{Case 2: \(0<s_0\leq s^*\)}} Notice that \(r(s_0 )\leq r(s^*)=1\). Therefore
  \[
    \int_{0}^{s_0 }\left( v^{*}(s)-u^*(s) \right) \, ds = \int_{0}^{s_0 }v^{*}(s)(1-r(s))\, ds \geq (1-r(s_0 ))\int_{0}^{s_0 }v^{*}(s)\, ds.
  \]
  By the same argument, we have
  \[
    \int_{s_0 }^{s^{*}}\left( v^{*}(s)-u^{*}(s) \right) 	\, ds \leq (1-r(s_0 ))\int_{s_0 }^{s^*}v^{*}(s)	\, ds.
  \]
  Using the previous inequalities we obtain
  \begin{align*}
    \int_{0}^{s^*}\left( v^{*}(s)-u^*(s) \right) \, ds & = \int_{0}^{s_0 }\left( v^*(s)-u^*(s) \right) \, ds + \int_{s_0 }^{s^*}\left( v^*(s)-u^*(s) \right) \, ds                                                                          \\
                                                       & \leq  \int_{0}^{s_0 }\left( v^*(s)-u^*(s) \right)\, ds + \int_{s_0 }^{s^*}v^*(s)\, ds \cdot \frac{\int_{0}^{s_0 }\left( v^*(s)-u^*(s) \right) \, ds }{\int_{0}^{s_0 }v^*(s)\, ds } \\
                                                       & =\int_{0}^{s_0 }\left( v^*(s)-u^*(s) \right) \, ds \cdot \frac{\int_{0}^{s^*}v^*(s)\, ds }{\int_{0}^{s_0 }v^*(s)\, ds }  \leq c(\alpha)\delta_{s_0 }
  \end{align*}
  Notice that
  \[
    \begin{split}
      \int_{0}^{s^{*}}\left( v^{*}(s)-u^{*}(s) \right)  \, ds 
       & = \int_{t^*}^{1}\left( \rho_0 (\tau)-\rho(\tau) \right) \, d\tau
    \end{split}
  \]
  where \(v^{*}(s^*)=t^*=u^{*}(s^{*}).\)
  By the discussion at the end of the proof of Lemma \ref{lm:the-faber-krahn-inequalities-lemma} we have
  \begin{align*}
    \rho (t) \leq  \rho_0 (t), \text{ for } t \geq t^* \quad \text{and} \quad
    \rho (t) \geq  \rho_0 (t), \text{ for } t \leq  t^*
  \end{align*}
  therefore,
  \[
    \int_{t^*}^{1}\left( \rho_0 (\tau)-\rho(\tau) \right) \, d\tau\geq \int_{T}^{1}\left( \rho_0(\tau)-\rho(\tau) \right) \, d\tau      = \int_{T}^{1}\rho_0 (\tau)\, d\tau =: \phi_{n,\alpha}(T),
  \]
  where $T=\max u\leq 1$.
  From the identity \(\rho_0 '(t)=-\frac{1}{\alpha t \Theta(\rho_0(t))}\) we can easily conclude that \(\rho_0 \) is a \(C^{\infty}\) function on \((0,1)\), and therefore, \(\phi \in C^{\infty}(0,1).\) Notice that \(\phi'(t)=-\rho_0 (t), \phi''(t)=-\rho_0 '(t)= \frac{1}{\alpha t \Theta(\rho_0 (t))},\)

  Our main concern will be to determine the rate at which \(\phi_{n,\alpha}(T)\) decreases as \(T \) goes to \(1\) from below.
  By the second equation in \eqref{eq:the-formula-for-k_c} we obtain that $k(r)\geq -\frac{2^n }{n}r, r \in [0,1]$ which implies that $\int_{0}^{r}k(\tau)\, d\tau\geq \frac{-r^2}{n}2^{n-1}$ for $r \in [0,1]$. Using the inequality $e^{-x}\geq 1-x,$ $x\geq 0$ we obtain that
  \begin{equation}\label{eq:auxiliary-estimate-for-bergman-weight-function-for-small-argument}
    \mathcal{W}_n^{\alpha}(r)=e^{\alpha\int_{0}^{r}k(\tau)\, d\tau }\geq e^{-\frac{\alpha r^2}{n}2^{n-1}}\geq 1-\alpha\frac{r^2}{n}2^{n-1},\quad \text{for } r \in [0,1].
  \end{equation}
  Notice that in that case for every $T \in [0,1]$ and for every $r \in [0,1]$ we have
  \[
    \begin{split}
      \phi_{n,\alpha}(T)
       & = \int_{\mathbb{R}^{n}}^{}\left( \mathcal{W}_n^{\alpha}(x)-T \right) _+\, dm_S(x)\geq \int_{|x|<r}^{}\left( \mathcal{W}_n^{\alpha}(x)-T \right) _+\, dm_S(x) \\
       & \geq \int_{|x|<r}\left( 1-\alpha 2^{n-1} \frac{r^2}{n}-T \right)_+ \, dm_S(x).
    \end{split}
  \]
  By a suitable choice of $0\leq r\leq 1$ we can achieve that the expression in the brackets is greater than or equal to $\frac{1-T}{2}$, i.e. using
  \[
    r= \sqrt{\frac{(1-T)}{\max\{1,\frac{\alpha2^{n}}{n}\}}}
  \]
  where we use $\max\{1,\frac{\alpha 2^{n}}{n}\}$ since $r\leq  1$; we obtain
  \[
    \begin{split}
      \phi_{n,\alpha}(T) & \geq \frac{1-T}{2} \int_{|x|<r}^{}\frac{2^n}{\left( 1+|x|^2 \right) ^{n}}\, dx\geq \frac{(1-T)}{2}r^{n}\omega_{n}                   \\
                         & =(1-T)^{\frac{n}{2}+1} \frac{\omega_n }{2 \left( \max\{1,\frac{\alpha2^n}{n}\} \right)^{n/2} } = C_{n,\alpha} (1-T)^{\frac{n}{2}+1}
    \end{split}
  \]
  where $C_{n,\alpha}=\frac{\omega_n }{2 \left( \max\{1,\frac{\alpha2^n}{n}\} \right)^{n/2} }$. Note that $C_{n,\alpha} $ is positive and continuous in $\alpha$.

  Notice that  \(L^2\left( \mathbb{R}^{n},\frac{\mathcal{W}_n^{\alpha}dm_S}{c(\alpha)} \right) \) is a Hilbert space with the scalar product given by
  \[
    \langle g_1 ,g_2  \rangle=\frac{1}{c(\alpha)} \int_{\mathbb{R}^{n}}^{}g_1 (x)\overline{g_2(x)} \mathcal{W}_n^{\alpha}(x)\, dm_S(x).
  \]
  Since \(\left\lVert f \right\rVert _{\alpha,2}=\left\lVert I_{x_0 } \right\rVert _{\alpha,2}=1\) we have
  \begin{align*}
    \left\lVert |f|-I_{x_0 } \right\rVert ^2  = 2 - 2 \langle |f|,I_{x_0 } \rangle
     & =2- \frac{2}{c(\alpha)} \int_{\mathbb{R}^{n}}^{}\left| f(x)\right|  \frac{\mathcal{W}_n^{\alpha/2}\left( \varphi_{x_0 }(x) \right) }{\mathcal{W}_n^{\alpha/2}(x)}\mathcal{W}_n^{\alpha}(x)\, dm_S(x) \\
     & \leq  2- 2 \left| f(y) \right| \mathcal{W}_n^{\alpha/2}\left( \varphi_{x_0 }(y) \right) \mathcal{W}_n^{\alpha/2}(y),
  \end{align*}
  where the last inequality holds for every \(y \in \mathbb{R}^{n}\) by Theorem \ref{tm:point-evaluation-theorem}. Taking \(y=x_0 \) we obtain \(\inf_{x_0 \in \mathbb{R}^{n}}\left\lVert |f|-I_{x_0 } \right\rVert^2 \leq 2-2 \sup_{x_0 \in \mathbb{R}^{n}} |f(x_0 )| \mathcal{W}_n^{\alpha/2}(x_0 )=2\left(1-\sqrt{T} \right)\leq 2(1-T),\) where \(T=\sup_{x \in \mathbb{R}^{n}}|f(x)|^2 \mathcal{W}_n^{\alpha}(x).\)

  Combining the previous results, we have
  \[
    \inf_{x_0 \in \mathbb{R}^{n}}\left\lVert |f|-I_{x_0 } \right\rVert \leq \sqrt{2(1-T)}\leq \sqrt{2}\frac{\left[ \phi_{n,\alpha}(T) \right] ^{\frac{1}{n+2}}}{C_{n,\alpha}^{\frac{1}{n+2}}}\leq \sqrt{2}\left[ \frac{c(\alpha) F_\alpha(s_0)}{C_{n,\alpha}} \right] ^{\frac{1}{n+2}} \delta(f;\Omega,\alpha)^{\frac{1}{n+2}}.
  \]

  Consider now the case when \(\limsup_{|x|\to \infty}u(x)< \infty.\) Let $f_\epsilon=f/\left\lVert f \right\rVert _{\alpha+\epsilon,2}$ and
  \[
    T_\epsilon=\sup_x \frac{|f(x)|^2 \mathcal{W}_n^{\alpha+\epsilon}(x)}{\left\lVert f \right\rVert _{\alpha+\epsilon,2}^2}.
  \]
  Then for every \(\epsilon >0, \) we have
  \begin{equation}\label{eq:stability-for-the-concentration-operator-auxiliary-inequality-for-T-epsilon}
    \sqrt{(1-T_\epsilon)} \leq  \left[ \frac{c(\alpha+\epsilon)F_{\alpha+\epsilon}(s_0)}{C_{n,\alpha+\epsilon}} \right] ^{\frac{1}{n+2}} \delta\left( f_{\epsilon}; \Omega,\alpha+\epsilon \right) ^{\frac{1}{n+2}}.
  \end{equation}
  One can see, by the dominated convergence theorem, that $C_{\Omega}^{\alpha+\epsilon}(f_\epsilon)\to C_{\Omega}^{\alpha}(f)$ and $C_{\mathbb{B}}^{\alpha+\epsilon}(\mathbf{1}) \to C_{\mathbb{B}}^{\alpha}(\textbf{1})$ as $\epsilon \to 0$, hence, $\delta\left( f_\epsilon,\Omega,\alpha+\epsilon \right) \to \delta\left( f;\Omega,\alpha \right)$.
  Also, $\lim_{\epsilon \to 0} \frac{c(\alpha+\epsilon)F_{\alpha+\epsilon}(s_0 )}{C_{n,\alpha+\epsilon}}= \frac{c(\alpha)F_{\alpha}(s_0 )}{C_{n,\alpha }}$. It is straightforward to see that $\lim_{\epsilon \to 0} T_\epsilon = T$; therefore, taking the limit as $\epsilon \to 0$ in \eqref{eq:stability-for-the-concentration-operator-auxiliary-inequality-for-T-epsilon} we obtain
  \[
    \sqrt{1-T}\leq \left[ \frac{c(\alpha)F_{\alpha}(s_0 )}{C_{n,\alpha}} \right]^{\frac{1}{n+2}} \delta(f;\Omega,\alpha)^{\frac{1}{n+2}}.
  \]
  Therefore, we have
  \[
    \inf_{x_0 \in \mathbb{R}^{n}}\left\lVert |f|-I_{x_0 } \right\rVert \leq \sqrt{2(1-T)}\leq \sqrt{2}\left[ \frac{c(\alpha) F_{\alpha}(s_0)}{C_{n,\alpha}} \right] ^{\frac{1}{n+2}} \delta(f;\Omega,\alpha)^{\frac{1}{n+2}},
  \]
  which finishes the proof.
\end{proof}
\section{Proof of Theorem \ref{tm:stability-convex-functional}}
\begin{proof}
  Notice that by Theorem \ref{tm:point-evaluation-theorem} we have \(T\leq 1.\)
  Denote \(v=\mathcal{W}_n^{\alpha}\) and \(u=|f|^{p}\mathcal{W}_n^{\alpha}.\) We will follow the arguments from \cite{arXiv:2412.10940}. First, we decompose \(G= G_1+G_2\), where
  \[
    G_1(t)=\left\{\begin{array}{l l}
      G(t),                           & 0\leq t\leq T, \\
      G_-'(T)\left( t-T \right)+G(T), & T\leq t\leq 1.
    \end{array}\right.
  \]
  Note that \(G_1\) is convex itself. By Theorem \ref{tm:convex-functional-theorem}, we have
  \begin{align*}
         & \int_{\mathbb{R}^{n}}^{}G(v)\, dm_S(x)- \int_{\mathbb{R}^{n}}^{}G(u)\, dm_S                                                                                  \\
    =    & \int_{\mathbb{R}^{n}}^{}G_1(v)\, dm_S-\int_{\mathbb{R}^{n}}^{}G_1(u)\, dm_S + \int_{\mathbb{R}^{n}}^{}G_2(v)\, dm_S  - \int_{\mathbb{R}^{n}}^{}G_2(u)\, dm_S \\
    \geq & \int_{\mathbb{R}^{n}}^{}G_2(v)\, dm_S - \int_{\mathbb{R}^{n}}^{}G_2 (u)\, dm_S                                                                               \\
    =    & \int_{T}^{1}\left( G'(t)-G'_-(T) \right) \rho_0(t)\, dt,
  \end{align*}
  where the last equality holds because \(\rho(t)=0\) for \(t \geq T.\)
\end{proof}
\section{Appendix}\label{sec:appendix}
In this section, we prove a point evaluation estimate for log-subharmonic functions in the Bergman space on the unit sphere, the same as in the case of the Bergman spaces on the hyperbolic ball. We also give the optimal constant.

In the hyperbolic ball the M\"obius map \(\varphi_z\) is the map with the properties: \(\varphi_z(0)=z\) and \(\varphi_z^{-1}=\varphi_z.\) Here we define maps with the same properties.
\begin{lemma}
  For every \(\xi \in \mathbb{S}^{n}\) there exists a symmetric orthogonal matrix \(A\) such that \(Ae_{n+1}=\xi.\)
\end{lemma}
\begin{proof} Left as an easy exercise.
\end{proof}
Let us define \(\psi_\xi x=Ax,\) where the matrix \(A\) is given by the previous lemma. Then we have  \(\psi_\xi(e_{n+1})=\xi\) and \(\psi_{\xi}^{-1}= \psi_\xi.\) Also, for every \(x_0 \in \mathbb{R}^{n}\) we will use the notation \(\varphi_{x_0 } = S^{-1}\circ \psi_\xi \circ S,\) where \(\xi=S(x_0 ).\)

Let \(\psi : \mathbb{S}^{n}\to \mathbb{S} ^{ n}\) be given by \(\psi(x)=Ax\) for a symmetric, orthogonal \((n+1) \times (n+1)\) matrix $A$ with positive determinant. Denote \(dV=dx_1 \wedge dx_2 \wedge \dots \wedge dx_{n+1}\) and \(dV_S= \iota_N dV\) be the standard Riemannian volume form on \(\mathbb{S}^{n},\) where \(N=\sum_{k=1}^{n+1} x_k \frac{\partial }{\partial x_k}.\) By using \cite[Lemma 9.11]{JohnMLee2002}, or by  a straightforward calculation one can show that \(\psi^* dV=dV. \) Therefore, for \(X_2 ,X_3 ,\dots ,X_{n+1} \in T_p\mathbb{S}^{n}\) we have
\begin{equation*}
  \psi^{*}\left( \iota_{N_{\psi(p)}dV_{\psi(p)}} \right) = \iota_{\psi_{*}^{-1}N_{\psi(p)}}\psi^*dV_{\psi(p)}=\iota_{N_p}dV_p.
\end{equation*}
Hence, we have \(\psi^{*}dV_S=dV_S\) and
\[
  \int_{\mathbb{S}^{n}}^{}f \, dV_S = \int_{\mathbb{S}^{n}}^{}f\, \psi^{*}\left( dV_S \right) = \int_{\mathbb{S}^{n}}^{}\psi^{*}\left( f \circ \psi\,  dV_S \right) = \int_{\mathbb{S}^{n}}f\circ \psi\, dV_S
\]
for every smooth function \(f : \mathbb{S}^{n} \to \mathbb{R},\) see \cite[Proposition 10.20 d)]{JohnMLee2002}.

In the case of a negatively oriented matrix \(A\) we have \(\psi^{*}dV_S=-dV_S.\) It can be proved as in the previous case, taking into account that \(\psi_* N_p=N_{\psi(p)}\). Hence
\[
  \int_{\mathbb{S}^{n}}^{}f \, dV_S = -\int_{\mathbb{S}^{n}}^{}f\, \psi^{*}\left( dV_S \right) = -\int_{\mathbb{S}^{n}}^{}\psi^{*}\left( f \circ \psi\,  dV_S \right) = \int_{\mathbb{S}^{n}}f\circ \psi\, dV_S .
\]
If \(S: \mathbb{R}^{n}\to \mathbb{S}^{n}\) is the inverse stereographic projection (see \nameref{sec:introduction}), we get
\[
  \int_{\mathbb{R}^{n}}^{}f\left( \psi\left( S(x) \right)  \right) \, \frac{2^{n}dx}{\left( 1+x^2 \right) ^{n}}= \int_{\mathbb{R}^{n}}^{}f\left( S (x)\right) \, \frac{2^{n}dx}{\left( 1+x^2 \right) ^{n}},
\]
for every orthogonal \(\psi\) satisfying \(\psi \circ \psi = Id.\)
Let \(\varphi=S^{-1} \circ \psi \circ S.\) Then \(S\circ \varphi = \psi \circ S,\) therefore, we have
\[
  \int_{\mathbb{R}^{n}}^{}f\left( S\left( \varphi(x) \right)  \right) \, \frac{2^{n}\, dx}{\left( 1+x^2 \right) ^{n}} = \int_{\mathbb{R}^{n}}^{}f\left( S(x) \right) \, \frac{2^{n}\, dx}{\left( 1+x^2 \right) ^{n}}.
\]
The change of variable formula and the identity \(\varphi^2=\mathrm{Id}\) give us
\[
  \int_{\mathbb{R}^{n}}^{}f\left( S(x) \right) \left| \det \varphi'(x) \right| \, \frac{2^{n}\, dx}{\left( 1+\left| \varphi(x) \right| ^2 \right) ^{n} } = \int_{\mathbb{R}^{n}}^{}f\left( S(x) \right) \, \frac{2^{n}\, dx}{\left( 1+x^2 \right) ^{n}}.
\]
Since this holds for every smooth function \(f:\mathbb{S}^{n}\to \mathbb{R}\) we have
\[
  \left| \det \varphi'(x) \right| =\frac{\left( 1+\left| \varphi(x) \right| ^2 \right) ^{n}}{\left( 1+|x|^2 \right) ^{n}},
\]
for every \(x \in \mathbb{R}^{n}\) where \(\varphi \) is defined (which is \(\mathbb{R}^{n}\setminus \{a\}\) for some \(a\)).

\subsection{Submean-value property of \texorpdfstring{\(\Delta_S\)}{DeltaS}-subharmonic functions}\label{subsec:mean-value property}
Here, we will prove the mean-value property of \(\Delta_S\)-subharmonic functions. For a function \(f\) on \(\mathbb{R}^{n}\) we say that it is radial if \(f(x)=g(|x|)\) for some function \(g\) defined on \([0,\infty).\) For a continuous function we define the radialization \(f^{\#}\) of \(f\) by
\[
  f^{\#}(x)=\int_{\mathbb{S}^{n-1}}^{}f(|x|\zeta)\, d\sigma_{n-1}(\zeta)
\]
where \(\sigma_{n-1}\) is the normalized measure on the unit sphere \(\mathbb{S}^{n-1}.\)  By \(\omega_n\) we denote the volume of the unit ball in \(\mathbb{R}^{n}.\) Note that the surface area of \(\mathbb{S}^{n-1}\) is then equal to \(n \omega_n.\)

\begin{lemma}
  If \(f \in C^{2}(\mathbb{B}_R),\) then for all \(0<r<R\) we have
  \begin{enumerate}
    \item[a)] \(\frac{d}{dr}\int_{\mathbb{S}^{n-1}}^{}f(r \zeta)\, d \sigma_{n-1}(\zeta) = \frac{4}{n \omega_n}r^{1-n}\left( 1+r^2 \right) ^{n-2}\int_{B_r}^{}\Delta_S f(x) \, \frac{dx}{\left( 1+x^2 \right) ^{n}}; \)
    \item[b)] \(f(0)=\int_{\mathbb{S}^{n-1}}^{}f(r \zeta)\, d \sigma_{n-1}(\zeta) - \int_{B_r}^{}  g(|x|,r)\Delta_S f(x)\, \frac{dx}{\left( 1+x^2 \right) ^{n}}, \)\\
          where \(g(|x|,r)=\frac{4}{n \omega_n}\int_{|x|}^{r}\frac{(1+s ^2)^{n-2}}{s^{n-1}}\, ds.\)
  \end{enumerate}
\end{lemma}
\begin{proof}
  Since \(\left( \Delta_S f \right) ^{\#}= \Delta_S \left( f^{\#} \right) \) it is enough to prove a) for radial functions.
  Notice that for \(f(x)=u(r^2)\) we have
  \[
    \Delta_S f(x)=r^2 \left( 1+r^2 \right) ^{2}u''(r^2)+\frac{n}{2}u'(r^2)\left( 1+r^2 \right) ^2 + (2-n)(1+r^2) u'(r^2) r^2.
  \]
  Also, if we let \(v(t)=t ^{\frac{n}{2}}\left( 1+t \right) ^{2-n}u'(t)\) we have \(\Delta_S f(x)=r^{2-n}\left( 1+r^2 \right) ^{n}v'(r^2).\) Therefore,
  \begin{align*}
    \int_{B_\rho}^{}\Delta_S f \, dm_S= & 2^{n} n \omega_n\int_{0}^{\rho}r^{n-1}\Delta_S f(r)\,   \frac{dr}{\left( 1+r^2 \right) ^{n}}
    =                                    2^{n}n \omega_n\int_{0}^{\rho}v'(r^2)r\, dr                                                   \\
    =                                   & 2^{n-1}n \omega_n \rho^{n}(1+\rho^2)^{2-n} u'(\rho^2).
  \end{align*}
  On the other hand, we have \(\frac{d}{d \rho} \int_{\mathbb{S}^{n-1}}^{}f(\rho \xi)\, d \sigma_{n-1}(\xi) = 2 \rho u'(\rho^2).\) Integrating the equality from a) we get b).
\end{proof}
\begin{corollary}\label{tm:the-mean-value-property-theorem}
  Let \(U \subset \mathbb{R}^{n}\) be an open set. A real-valued \(C^{2}\) function is subharmonic with respect to \(\Delta_S\) on \(U\) if and only if  \(\Delta_S f \geq 0\) on \(U.\) Also, \(f \in C^{2}(U)\) is \(\Delta_S\)-harmonic if and only if \(\Delta_Sf=0\) on \(U.\)
\end{corollary}
\begin{proof} Follows easily from the previous lemma and the  \(\varphi_{x_0}\)-invariance of the  Laplacian. \qedhere

\end{proof}
\begin{theorem}
  Let \(f\) be a \(\Delta_S\)-subharmonic function on an open connected  set \(U.\) If \(f\) attains a global maximum on \(U\) then \(f\) is constant. Also, if $U$ is relatively compact, \(f \in C\left( \overline{U} \right) \) and  \(f \leq 0\) on \(\partial U\neq \emptyset,\) then \(f \leq 0\) on \(U.\)
\end{theorem}
The proof follows from the well-known idea for the same statement in the complex plane with the standard Laplace operator.
\begin{remark}\label{rm:the-dirichlet-problem-on-the-sphere}
  In \cite{Symeonidis2003} it is proven that for every geodesic ball \(B_r(x_0)\subsetneq \mathbb{S}^{n}\) and every \(f \in C\left( \partial B_r(x_0) \right) \) the Dirichlet problem
  \begin{align*}
    \Delta_S u & = 0 \text{ on } B_r(x_0)           \\
    u          & = f \text { on } \partial B_r(x_0)
  \end{align*}
  has a unique solution \(u \in C^2(B_r(x_0)) \cap C\left( \overline{B_r(x_0 )} \right) .\) Also, in the same paper the exact formula for such a function \(u\) is given.
\end{remark}
\begin{theorem}
  Let \(f \in C(U)\) be a \(\Delta_S\)-subharmonic function. Then for every \(x_0 \in  U\) and every \(r >0\) such that \(\varphi_{x_0 }(\overline{\mathbb{B}_r(0)}) \subset U\) the inequality holds:
  \[
    f(x_0) \leq \int_{\mathbb{S}^{n-1}}^{}f\left( \varphi_{x_0 }\left( r \zeta \right)  \right) \, d\sigma_{n-1}(\zeta).
  \]
\end{theorem}
\begin{proof}
  Let \(x_0 \in U\) and \(r >0\) such that \(\varphi_{x_0 }(\overline{\mathbb{B}_r(0)}) \subset U.\) Since \(\varphi_{x_0 }\) is an isometry for the spherical metric, \(\varphi_{x_0}(\mathbb{B}_r(0))\) is a geodesic ball around \(x_0.\) Therefore, there exists a harmonic function \(u\) such that \(u = f\) on \(\partial \varphi_{x_0 }(\mathbb{B}_r(0)).\) The function \(f-u\) is a \(\Delta_S\)-subharmonic function and since \(f-u \leq 0\) on the boundary we have  \(f\leq  u\) on the \(\varphi_{x_0 }\left( \mathbb{B}_r(0) \right).\)  Using the mean-value property of the harmonic function \(u\) we have
  \[
    f(x_0) \leq  u(x_0) \leq  \int_{\mathbb{S}^{n-1}}^{}u(\varphi_{x_0}(r \zeta))\, d\sigma_{n-1}(\zeta) =  \int_{\mathbb{S}^{n-1}}^{}f(\varphi_{x_0}(r \zeta))\, d\sigma_{n-1}(\zeta)
  \]
  where the last equality holds since \(f=u\) on \(\partial \varphi_{x_0 }(\mathbb{B}_r(0)).\)
\end{proof}
\subsection{Point evaluation estimate }
\begin{theorem}\label{tm:point-evaluation-theorem}
  For  \(p >0 \) and every  \(f \in B_{\alpha,p}\) such that the function \(\log |f|\) is \(\Delta_S\)-subharmonic  and \(\limsup_{|x| \to \infty} |f(x)|^{p}\mathcal{W}_n^{\alpha}(x) < \infty\) the inequality
  \[
    \left| f(x_0 ) \right| ^{p}\mathcal{W}_n^{\alpha}(x_0 )\leq \frac{1}{c(\alpha)}\int_{\mathbb{R}^{n}}^{}\left| f(x) \right|^{p} \mathcal{W}_n^{\alpha}(x)\, dm_S(x),
  \]
  holds for every \(x_0 \in \mathbb{R}^{n}.\)
\end{theorem}
\begin{proof}
  The case $f \equiv 0$ is immediate.\\
  Let us first consider the case when \(\limsup_{|x| \to \infty} |f(x)|^{p}\mathcal{W}_n^{\alpha}(x)=0.\)
  Since \(\log|f|\) is \(\Delta_S\)-subharmonic and \(t\mapsto e^{pt}\) is increasing and convex for every \(p>0\), the function \(|f|^p\) is \(\Delta_S\)-subharmonic on $\{ f\neq 0\}$. At its zeros, the submean inequality follows from nonnegativity. Hence
  \[
    |f(0)|^p\leq
    \int_{\mathbb S^{n-1}}|f(r\zeta)|^p\,d\sigma_{n-1}(\zeta),
    \qquad r>0.
  \]
  Multiplying both sides by \(\mathcal{W}_n^{\alpha}(r)\frac{2^{n}n\omega_n r^{n-1} }{\left( 1+r^2 \right) ^{n}c(\alpha)},\) and integrating for \(r \in [0,\infty) \) we obtain
  \[
    |f(0)|^{p}\leq \frac{1}{c(\alpha)}\int_{\mathbb{R}^{n}}^{}\left| f(x) \right| ^{p}\mathcal{W}_n^{\alpha}(x)\, dm_S(x).
  \]
  Let \(x_0  \in \mathbb{R}^{n}\setminus\{0\}\) and \(g(x)=f\left( \varphi_{x_0}(x) \right) \frac{\left( \mathcal{W}_n\circ \varphi_{x_0} \right)^{\alpha/p} (x)}{\mathcal{W}_n^{\alpha/p}(x)}.\)  The function \(\log |g|\) has the submean-value property. Indeed, where $g\neq0$ and \(x \neq \frac{-x_0 }{|x_0 |^2}\), we have
  \[
    \log|g|=\log |f|\circ \varphi_{x_0 } + \frac{\alpha}{p}\left( \log \mathcal{W}_n \circ \varphi_{x_0 }-\log \mathcal{W}_n \right).
  \]
  The first term is $\Delta_S$-subharmonic by isometry invariance, and the expression in parentheses is $\Delta_S$-harmonic, since $\Delta_S (\log \mathcal{W}_n \circ \varphi_{x_0 }-\log \mathcal{W}_n)=0.$ Consequently, $|g|^p$ is $\Delta_S$-subharmonic whenever $g\neq0$ and $x\neq -\frac{x_0 }{|x_0 |^2}$.
  When \(x=-\frac{x_0 }{|x_0 |^2},\) we have \(\varphi_{x_0 }(x)=\infty.\) Hence, by the condition $\limsup_{|x| \to \infty} |f(x)|^p \mathcal{W}_n^{\alpha}(x)=0$ we can extend $g$ continuously by setting \(g(x)=\lim_{y \to x} g(y) =0.\) At every zero of $g$, including this pole, the submean inequality for $|g|^p$ follows from its nonnegativity.
  Using the submean-value property of \(\log |g|\) and the same idea from the beginning of the proof, we obtain
  \begin{align*}
    \left| f\left( x_0 \right)  \right| ^{p}\mathcal{W}_n^{\alpha}(x_0 )=\left| g(0)  \right| ^{p}\leq \left\lVert g \right\rVert _{\alpha,p}^{p} & =  \frac{1}{c(\alpha)}\int_{\mathbb{R}^{n}}^{}\left| f\left( \varphi_{x_0} (x) \right)  \right| ^{p} \mathcal{W}_n^{\alpha}\left( \varphi_{x_0 }(x) \right) \, \frac{2^{n}\, dx}{\left( 1+x^2 \right) ^{n}} \\
                                                                                                                                                  & = \frac{1}{c(\alpha)}\int_{\mathbb{R}^{n}}^{}\left| f\left( y \right)  \right| ^{p}\mathcal{W}_n^\alpha(y)\, \frac{2^{n}\,dy}{\left( 1+y^2 \right)^{n} },
  \end{align*}
  where the last formula follows from the change of coordinates with \(y = \varphi_{x_0}(x)\) and using the identity \(\left| \det \varphi_{x_0}'(x) \right| = \frac{\left( 1+\left| \varphi_{x_0}(x) \right| ^2 \right) ^{n}}{\left( 1+x^2 \right) ^{n}}.\)

  Let us now consider the case  \(\limsup_{|x| \to \infty} \left| f(x) \right| ^{p}\mathcal{W}_n^{\alpha}(x) < \infty.\) For every \(\epsilon>0\) we have \(\lim_{|x| \to \infty} |f(x)|^{p}\mathcal{W}_n^{\alpha+\epsilon}(x)=0\),  and therefore
  \[
    \left| f(x_0 ) \right| ^{p}\mathcal{W}_n^{\alpha+\epsilon}(x_0 )\leq \frac{1}{c(\alpha+\epsilon)}\int_{\mathbb{R}^{n}}^{}\left| f(x) \right|^{p} \mathcal{W}_n^{\alpha+\epsilon}(x)\, dm_S(x).
  \]
  Since  \(\mathcal{W}_n(x)\leq 1,\) for all \(x \in \mathbb{R}^{n}\) we have that \(|f(x)|^{p}\mathcal{W}_n^{\alpha+\epsilon}(x)\) increases as \(\epsilon \to 0.\) Using \(\lim_{\epsilon \to 0} c(\alpha+\epsilon)=c(\alpha)\) and the monotone convergence theorem we obtain the desired inequality by letting \(\epsilon\to 0+.\) Also, one can show that the function \(g\) is log-subharmonic even when \(\limsup_{|x| \to \infty} |f(x)|^{p}\mathcal{W}_n^{\alpha}\neq 0\) using the idea of removal an isolated bounded singularity of a subharmonic function (see \cite[Thm. 3.6.1]{ransford1995}).
\end{proof}
\textbf{Conflict of interest.} The authors declare that they have not conflict of interests.

\textbf{Funding.} The first author acknowledges partial financial support by the Ministry of Scientific and Technological Development and Higher Education grant 1259115. The second author is partially supported by MPNTR grant 174017, Serbia.

\textbf{Declaration of generative AI and AI-assisted technologies in the manuscript preparation process}
The first version of this manuscript posted on arXiv was prepared without the use of AI tools. During subsequent revisions, the authors used ChatGPT to assist with language editing and checking mathematical statements and arguments. The authors reviewed and edited the resulting content and take full responsibility for the content of the manuscript.
\nocite{*}
\printbibliography
\end{document}